\documentclass[review]{elsarticle}

\usepackage{amsmath} 
\usepackage{amsfonts}
\usepackage{hyperref}

\journal{Journal of \LaTeX\ Templates}

\newtheorem{theorem}{Theorem}
\newtheorem{corollary}{Corollary}
\newtheorem{lemma}{Lemma}
\newenvironment*{proof}
    {\begin{trivlist}\item[\hspace{\labelsep}\textit{Proof.}]}
    {\hspace*{\fill}\rule{0.5em}{0.5em}\end{trivlist}}

\begin{document}

\begin{frontmatter}

\title{The spectral norm of a Horadam circulant matrix}



\author[mymainaddress]{Jorma K. Merikoski}
\ead{jorma.merikoski@uta.fi}

\author[mymainaddress]{Pentti Haukkanen}
\ead{pentti.haukkanen@uta.fi}

\author[mysecondaryaddress]{Mika Mattila}
\ead{mika.mattila@tut.fi}

\author[mytertiaryaddress]{Timo Tossavainen\corref{mycorrespondingauthor}}
\cortext[mycorrespondingauthor]{Corresponding author}
\ead{timo.tossavainen@ltu.se}

\address[mymainaddress]{Faculty of Natural Sciences, FI-33014 University of Tampere, Finland}
\address[mysecondaryaddress]{Department of Mathematics, Tampere University of Technology, P.O. Box 553, FI-33101 Tampere, Finland}
\address[mytertiaryaddress]{Department of Arts, Communication and Education, Lulea University of Technology, SE-97187 Lulea, Sweden}

\begin{abstract}
Let $a$, $b$, $p$, $q$ be integers and~$(h_n)$ defined by $h_0=a$, $h_1=b$, 
$h_n=ph_{n-1}+qh_{n-2}$, $n=2,3,\dots$. Complementing to certain
previously known results, we study the spectral norm of
the circulant matrix corresponding to $h_0,\dots,h_{n-1}$.
\end{abstract}


\end{frontmatter}


\section{Introduction}

Throughout this paper, let $a,b,p,q\in\mathbb{Z}$. We
define the {\em Horadam sequence} $(h_n)=(h_n(a,b;p,q))$ via
\begin{eqnarray*}
h_0=a,\quad h_1=b,\qquad\qquad
\\
h_n=ph_{n-1}+qh_{n-2},\quad n=2,3,\dots.
\end{eqnarray*}
We also use the following abbreviations:
\begin{enumerate}
\item[] $(f_n)=(h_n(0,1;1,1))$, the Fibonacci sequence;
\item[] $(\tilde{f}_n)=(h_n(0,1;p,q))$, a generalization of the
Fibonacci sequence; 
\item[] $(l_n)=(h_n(2,1;1,1))$, the Lucas sequence;
\item[] $(\tilde{l}_n)=(h_n(2,p;p,q))$, a generalization of the
Lucas sequence. 
\end{enumerate}
Some references call $(\tilde{l}_n)$ the Lucas sequence. In order to keep the language simple, 
we follow the custom in \cite[p. 8]{Ko} and call the sequence of Lucas numbers briefly the Lucas sequence.

For $n \ge 1$, we write
\begin{eqnarray*}
{\bf f}=(f_0,\dots,f_{n-1}),\quad \tilde{{\bf f}}=(\tilde{f}_0,\dots,\tilde{f}_{n-1}),
\\
{\bf l}=(l_0,\dots,l_{n-1}),\quad \tilde{{\bf l}}=(\tilde{l}_0,\dots,\tilde{l}_{n-1}),\,\,\,
\\
{\bf h}=(h_0,\dots,h_{n-1}).\qquad\qquad
\end{eqnarray*}

Let ${\bf x}=(x_0,\dots\,x_{n-1})\in\mathbb{R}^n$.
The corresponding circulant matrix~$\bf C(x)$ is defined as
$$
\bf C(x)=\left(\begin{array}{ccccc}
x_0&x_1&\dots&x_{n-2}&x_{n-1}\\
x_{n-1}&x_0&\dots&x_{n-3}&x_{n-2}\\
x_{n-2}&x_{n-1}&\dots&x_{n-4}&x_{n-3}\\
\vdots&\vdots&\vdots&\vdots&\vdots\\
x_2&x_3&\dots&x_0&x_1\\
x_1&x_2&\dots&x_{n-1}&x_0
\end{array}\right).
$$
We let $\|\cdot\|$ stand for
the spectral norm.
Our problem is to compute $\|\bf C(h)\|$ under suitable assumptions. Recently, Kocer et al.~\cite{KMT},
\. Ipek~\cite{Ip}, Liu~\cite{Li}, and Bah\c si~\cite{Ba} have already studied this question.
We will survey their results in Section~\ref{previous} and give further results in
Sections~\ref{new1} and~\ref{new2}.
Finally, we will complete our paper with some remarks in Section~\ref{remarks}.

\section{Previous results}
\label{previous}

Let us first study the eigenvalues and singular values of~$\bf C(x)$.

\begin{theorem}
The eigenvalues of~$\bf C(x)$ are
$$
\lambda_i=\sum_{j=0}^{n-1}x_j\omega^{-ij},\quad i=1,\dots,n,
$$
where $\omega$ is the $n$'th primitive root of unity.
\end{theorem}
\begin{proof}
See~\cite[Theorem~3.2.2]{Da}.
\end{proof}

\begin{corollary}
The singular values of~$\bf C(x)$ are
$$
\sigma_i=\Big|\sum_{j=0}^{n-1}x_j\omega^{-ij}\Big|,\quad i=1,\dots,n.
$$
Therefore
$$
\|{\bf C(x)}\|=\max_{1\le i\le n}\Big|\sum_{j=0}^{n-1}x_j\omega^{-ij}\Big|.
$$
\end{corollary}
\begin{proof}
Since $\bf C(x)$ is
normal, its singular values are the absolute values of eigenvalues.
\end{proof}
Applying this corollary, Kocer et al. \cite[Theorem~2.2]{KMT} proved that
$$
\|{\bf C(h)}\|=\max_{0\le i\le n-1}\Big|\frac{h_n+(pa-b+qh_{n-1})\omega^{-i}-a}
{q\omega^{-2i}+p\omega^{-i}-1}\Big|.
$$
The maximization problem restricts the use of this formula. The same authors also
proved  \cite[Corollary~2.3]{KMT} that
\begin{eqnarray}
\label{kocer}
\|{\bf C(h)}\|=\frac{h_n+qh_{n-1}+(p-1)a-1}{p+q-1},
\end{eqnarray}
assuming that $p,q\ge 1$ and $b=1$. Doing so, they suppose nothing on~$a$, but apparently
$a\ge 0$ must hold. (To see this, take $n=1$.)

Further, \. Ipek \cite[Theorem~1]{Ip} proved (independently of~(\ref{kocer})) that
$$
\|{\bf C(f)}\|=f_{n+1}-1
$$
and \cite[Theorem~2]{Ip}
$$
\|{\bf C(l)}\|=f_{n+2}+f_n-1.
$$

Liu~\cite[Theorem~9]{Li} extended~(\ref{kocer}) to
\begin{eqnarray}
\label{liu}
\|{\bf C(h)}\|=\frac{h_n+qh_{n-1}+(p-1)a-b}{p+q-1},
\end{eqnarray}
whenever $p+q\ne 1$, and to
$$
\|{\bf C(h)}\|=\frac{qh_{n-1}+(n-1)(qa+b)+a}{q+1},
$$
as $p+q=1$,
but assumed nothing about $a,b,p,q$.

Bah\c si \cite[Theorem~2.1]{Ba} proved (independently of~(\ref{liu})) that, if $p,q\ge 1$, then
$$
\|{\bf C}(\tilde{{\bf f}})\|=\frac{\tilde{f}_n+q\tilde{f}_{n-1}-1}{p+q-1}
$$
and \cite[Theorem~2.2]{Ba}
$$
\|{\bf C}(\tilde{{\bf l}})\|=\frac{\tilde{l}_n+q\tilde{l}_{n-1}+p-2}{p+q-1}.
$$

\section{Computation of $\|\bf C(h)\|, h\ge 0$}
\label{new1}

We first take a more general viewpoint and verify a theorem that applies also to other matrices than circulant ones or those having elements from a recurrence sequence.
If a matrix~$\bf A$ and a vector~$\bf x$ are entrywise
nonnegative (respectively, positive), we denote $\bf A\ge O$ and $\bf x\ge 0$
(respectively, $\bf A>O$ and $\bf x>0$).
We let $\lambda(\bf A)$ denote the Perron root of
a square matrix $\bf A\ge O$.

\begin{theorem}
\label{nonneg}
Assume that an $n\times n$ matrix $\bf A\ge O$ has all row sums and column sums
equal; let~$s$ be their common value. Then $\lambda({\bf A)=\|A}\|=s$.
\end{theorem}
\begin{proof}
Denoting ${\bf e}=(1,\dots,1)\in\mathbb{R}^n$, we have ${\bf Ae}={\bf A}^T{\bf e}=s\bf e$. So, $s$
is an eigenvalue of $\bf A$ and~${\bf A}^T$, and $\bf e$ is a corresponding eigenvector.
Since $\bf e>0$, actually $s=\lambda({\bf A})=\lambda({\bf A}^T)$, see
\cite[Theorem~8.3.4]{HJ}. Because
$$
{\bf A}^T{\bf Ae=A}^Ts{\bf e}=s{\bf A}^T{\bf e}=s^2\bf e,
$$
we similarly see that $s^2=\lambda({\bf A}^T{\bf A})=\|{\bf A}\|^2$.
\end{proof}

\begin{corollary}
\label{circx}
If ${\bf x}=(x_0,\dots,x_{n-1})\ge\bf 0$, then
$$
\|{\bf C(x)}\|=x_0+\dots+x_{n-1}.
$$
\end{corollary}
In order to apply this corollary in the case ${\bf x=h}$, we must compute
$h_0+\dots+h_{n-1}$.

\begin{lemma}
\label{hsum}
If $p+q\ne 1$, then
\begin{eqnarray}
\label{hsum1}
h_0+\dots+h_{n-1}=\frac{h_n+qh_{n-1}+(p-1)a-b}{p+q-1}.
\end{eqnarray}
If $p+q=1$ and $p\ne 2$, then
\begin{eqnarray}
\label{hsum2}
h_0+\dots+h_{n-1}=\frac{qh_{n-1}+(n-1)(qa+b)+a}{q+1}.
\end{eqnarray}
If $p=2$ and $q=-1$, then
\begin{eqnarray}
\label{hsum3}
h_0+\dots+h_{n-1}=n\,\frac{h_{n-1}+a}{2}.
\end{eqnarray}
\end{lemma}
\begin{proof}
Claim~(\ref{hsum1}) is equivalent to \cite[Equation~(3.5)]{Ho} and
to \cite[Lemma~5(1)]{Li}.
Claim~(\ref{hsum2}) is equivalent to \cite[Lemma~5(2)]{Li}.
Claim~(\ref{hsum3}) is
trivial, because the sequence~$(h_n)$ is arithmetic.
\end{proof}
We have now proved the following theorem.

\begin{theorem}
\label{thm}
If $\bf h\ge 0$, then
$$
\|{\bf C(h)}\|=h_0+\dots+h_{n-1},
$$
where $h_0+\dots+h_{n-1}$ is as in Lemma~{\rm \ref{hsum}}.
\end{theorem}

\section{Generalization of Theorem~\ref{thm}}
\label{new2}

Can the assumption $\bf h\ge 0$ be weakened?
Again, we begin by taking a more general viewpoint. For $m\in\mathbb{Z}$, we set
$$
m_n=m-\Big\lfloor\frac{m}{n}\Big\rfloor n.
$$

\begin{theorem}
Let $x=(x_0,\dots,x_{n-1})\in\mathbb{R}^n$. If
$$
\sum_{i=0}^{n-1}x_ix_{(i+j-1)_n}\ge 0
$$
for all $j=1,\dots,n$, then
\begin{eqnarray}
\label{claim}
{\|\bf C(x)\|}=|x_0+\dots+x_{n-1}|.
\end{eqnarray}
\end{theorem}
\begin{proof}
Write ${\bf B}=(b_{ij})={\bf C(x)}^T\bf C(x)$.
Letting ${\bf c}_1,\dots,{\bf{c}}_n$ to denote the column vectors of~$\bf C(x)$,
we have
$$
b_{1j}={\bf c}_1\cdot{\bf c}_j=\sum_{i=0}^{n-1}x_ix_{(i+j-1)_n}
$$
for all $j=1,\dots,n$. So, the first row of~$\bf B$ is nonnegative.
Summing its elements gives us
$$
r_1=\sum_{j=1}^n\sum_{i=0}^{n-1}x_ix_{(i+j-1)_n}=
\sum_{i=0}^{n-1}x_i\sum_{j=1}^nx_{(i+j-1)_n}=
\Big(\sum_{i=0}^{n-1}x_i\Big)^2.
$$
The last equation follows from the fact that
$$
\big\{i_n,\dots, (i+n-1)_n\big\}=\{0,\dots,n-1\}
$$
for all $i=0,\dots,n-1$.

A simple modification of the above reasoning applies to all rows of~$\bf B$. Consequently,
$\bf B\ge O$ with row sums
$$
r_1=\dots=r_n=\Big(\sum_{i=0}^{n-1}x_i\Big)^2.
$$
Since $\bf B$ is symmetric, every of its column sums has this value, too.
Applying Theorem~\ref{nonneg} to~$\bf B$, we therefore obtain
$$
\|{\bf C(x)}\|^2=\lambda({\bf B})=\Big(\sum_{i=0}^{n-1}x_i\Big)^2,
$$
and (\ref{claim}) follows.
\end{proof}

\begin{corollary}
If
\begin{eqnarray}
\label{assumption}
\sum_{i=0}^{n-1}h_ih_{(i+j-1)_n}\ge 0
\end{eqnarray}
for all $j=1,\dots,n$, then\
$$
\|{\bf C(h)}\|=|h_0+\dots+h_{n-1}|,
$$
where $h_0+\dots+h_{n-1}$ is as in Lemma~{\rm \ref{hsum}}.
\end{corollary}

\section{Concluding remarks}
\label{remarks}

In Section 2, we saw that, in the previous literature, $\|\bf C(h)\|$ is computed under various assumptions on $\bf h$. For example, in \cite{KMT}, the Horadam numbers were involved requiring that $a\ge 0$, $b=1$ and $p,q\ge1$.
We assumed first only that $\bf h\ge 0$, and then, even more generally, that~(\ref{assumption}) holds. As byproducts, Corollary 2 and Theorem 4 provided us with the corresponding results on~$\|\bf C(x)\|$, too.

We also mention that Yazlik and Taskara~\cite{YT} defined  the notion of 
a generalized $k$-Horadam sequence~$(H_{k,n})_{n\in\mathbb{N}}$. 
In fact, Liu~\cite{Li} ended up with (2) by studying a
circulant matrix corresponding to such a sequence. However, since
$k$ is fixed in \cite[Definition~1]{YT}, this sequence is nothing but
an ordinary Horadam sequence $(h_n)=(h_n(a,b;p,q))$ with $p=f(k)$ and $q=g(k)$.

\section*{References}

\bibliography{mybibfile}

\end{document}